\documentclass[letterpaper, 10 pt, conference]{ieeeconf}  

\IEEEoverridecommandlockouts                              

\usepackage[pdftex]{graphicx}          

\usepackage{epsfig} 
\usepackage{amsfonts}
\usepackage{amsmath,amssymb}
\usepackage{comment}
\usepackage{color}
\usepackage{cases}
\usepackage{ascmac}
\usepackage[subrefformat=parens]{subcaption}
\usepackage{cuted}
\usepackage{enumerate}
\usepackage{cite}

\usepackage{stmaryrd}

\allowdisplaybreaks[1]

\makeatletter
\renewcommand{\p@enumi}{A}
\makeatother

\newtheorem{problem}{Problem}
\newtheorem{proposition}{Proposition}

\newtheorem{lemma}{Lemma}
\newtheorem{theorem}{Theorem}

\newcommand{\calG}{{\mathcal G}}    
\newcommand{\calH}{{\mathcal H}}

\newcommand{\calT}{{\mathcal T}}

\newcommand{\calX}{{\mathcal X}}

\newcommand{\bbR}{{\mathbb R}}

\newcommand{\bbX}{{\mathbb X}}    
    
\newcommand{\bbZ}{{\mathbb Z}}

\newcommand{\sfd}{{\sf d}}
\newcommand{\sfe}{{\sf e}}

\newcommand{\bbra}[1]{\ensuremath{[\![#1]\!]} }  

\newcommand{\dhil}{{d_\calH}}

\newcommand{\diagbox}{\boxbslash}

\newcommand{\magenta}[1]{{\color{black}{#1}}} 
\newcommand{\blue}[1]{{\color{black}{#1}}} 
\newcommand{\rev}[1]{{\color{black}{#1}}} 
\newcommand{\fin}[1]{{\color{black}{#1}}} 
\newcommand{\red}[1]{{\color{black}{#1}}} 
\newcommand{\modi}[1]{{\color{black}{#1}}} 

\title{\LARGE \bf
Sinkhorn MPC: Model predictive optimal transport over\\dynamical systems*
}

\author{Kaito Ito$^{1}$, {\it Student Member, IEEE}, and Kenji Kashima$^{1}$, {\it Senior Member, IEEE}
\thanks{*This work was supported in part by \fin{the joint project of Kyoto University and Toyota Motor Corporation, titled ``Advanced Mathematical Science for Mobility Society'', and} by JSPS KAKENHI Grant Numbers JP21J14577, JP21H04875.}
\thanks{$^{1}$K. Ito and K. Kashima are with the Graduate School of Informatics, Kyoto University,
	Kyoto, Japan
        {\tt\small ito.kaito@bode.amp.i.kyoto-u.ac.jp; kk@i.kyoto-u.ac.jp}}%
\thanks{$ \copyright $ 2022 IEEE. Personal use of this material is permitted. Permission from IEEE must be obtained for all other uses, in any current or future media, including reprinting/republishing this material for advertising or promotional purposes, creating new collective works, for resale or redistribution to servers or lists, or reuse of any copyrighted component of this work in other works.}
}

\begin{document}

\maketitle
\thispagestyle{empty}
\pagestyle{empty}

\begin{abstract}
We consider the optimal control problem of steering an agent population to a desired distribution over an infinite horizon.
This is an optimal transport problem over a dynamical system, which is challenging due to its high computational cost.
In this paper, we propose {\em Sinkhorn MPC}, which is a dynamical transport algorithm combining model predictive control and the so-called Sinkhorn algorithm.
The notable feature of the proposed method is that it achieves cost-effective transport in real time by performing control and transport planning simultaneously.
In particular, for linear systems \fin{with an energy cost}, we reveal the fundamental properties of Sinkhorn MPC such as ultimate boundedness and asymptotic stability.
\end{abstract}

\section{Introduction}\label{sec:intro}
The problem of controlling a large number of agents has become a more and more important area in control theory with a view to applications in sensor networks, smart grids, intelligent transportation systems, and systems biology, to name a few.
One of the most fundamental tasks in this problem is to stabilize a collection of agents to a desired distribution shape with minimum cost.
This can be formulated as an optimal transport (OT) problem~\cite{Villani2003} between the empirical distribution based on the state of the agents and the target distribution over a dynamical system.

In \cite{Chen2017,Chen2018,Bakshi2020,de2021}, infinitely many agents are represented as a probability density of the state of a single system, and the problem of steering the state from an initial density to a target one with minimum energy is considered.
Although this approach can avoid the difficulty due to the large scale of the collective dynamics, in this framework, the agents must have the identical dynamics, and they are considered to be indistinguishable. 
In addition, even for linear systems, the density control requires us to solve a nonlinear partial differential equation such as the Monge-Amp\`{e}re equation or the Hamilton-Jacobi-Bellman equation.
\magenta{Furthermore, it is difficult to incorporate state constraints due to practical requirements, such as safety, into the density control.}

With this in mind, we rather deal with the collective dynamics directly without taking the number of agents to infinity.
\rev{This straightforward approach is investigated in multi-agent assignment problems; see e.g., \cite{Yu2014,Mosteo2017} and references therein.
In the literature, homogeneous agents following single integrator dynamics and easily computable assignment costs, e.g., distance-based cost, are considered in general.
On the other hand, for more general dynamics,} the main challenge of the assignment problem is the large computation time for obtaining the minimum cost of stabilizing each agent to each target state (i.e., point-to-point optimal control (OC)) and optimizing the destination of each agent.
This is especially problematic for OT in dynamical environments where control inputs need to be determined immediately for given initial and target distributions.

For the point-to-point OC, the concept of model predictive control (MPC) solving a finite horizon OC problem instead of an infinite horizon OC problem is effective in reducing the computational cost \magenta{while incorporating constraints}~\cite{Mayne2014}.
On the other hand, in \cite{Cuturi2013}, \red{several favorable computational properties of an entropy-regularized version of OT are highlighted. In particular,} entropy-regularized OT can be solved efficiently via the Sinkhorn algorithm.
Based on these ideas, we propose a dynamical transport algorithm combining MPC and the Sinkhorn algorithm, which we call {\em Sinkhorn MPC}.
Consequently, the computational effort can be reduced substantially.
Moreover, for linear systems \fin{with an energy cost}, we reveal the fundamental properties of Sinkhorn MPC such as ultimate boundedness and asymptotic stability.

\textit{Organization:}
The remainder of this paper is organized as follows. In Section~\ref{sec:optimal_transport}, we introduce OT between discrete distributions. In Section~\ref{sec:formulation}, we provide the problem formulation.
In Section~\ref{sec:Sinkhorn_MPC}, we describe the idea of Sinkhorn MPC.
\modi{In Section~\ref{sec:example}, numerical examples illustrate the utility of the proposed method.
In Section~\ref{sec:linear_system}, for linear systems with an energy cost, we investigate the fundamental properties of Sinkhorn MPC.}
Some concluding remarks are given in Section~\ref{sec:conclusion}.

\textit{Notation:}
Let $ \bbR $ denote the set of real numbers.
The set of all positive (resp. nonnegative) vectors in $ \bbR^n $ is denoted by $ \bbR_{> 0}^n $ (resp. $ \bbR_{\ge 0}^n $). We use similar notations for the set of all real matrices $ \bbR^{m\times n} $ and integers $ \bbZ $, respectively.
The set of integers $ \{1,\ldots,N\} $ is denoted by $ \bbra{N} $.
The Euclidean norm is denoted by $ \|\cdot \| $.
For a positive semidefinite matrix $ A $, denote $\|x\|_A := (x^\top A x)^{1/2}$.
The identity matrix of size $n$ is denoted by $I_n$. The matrix norm induced by the Euclidean norm is denoted by $ \| \cdot \|_2 $.
\rev{For vectors $ x_1,\ldots,x_m \in \bbR^n $, a collective vector $ [x_1^\top \ \cdots \ x_m^\top]^\top \in \bbR^{nm} $ is denoted by $ [x_1 ; \ \cdots \ ; x_m] $.}
For $ A = [a_1  \ \cdots \ a_n] \in \bbR^{m\times n}$, we write $ {\rm vec} (A) := \rev{[a_1; \ \cdots \ ;a_n]} $.
For $ \alpha = [\alpha_1 \ \cdots \ \alpha_N ]^\top \in \bbR^N $, the diagonal matrix with diagonal entries $ \{\alpha_i\}_{i=1}^N $ is denoted by \rev{$ \alpha^\diagbox $}.
The block diagonal matrix with diagonal entries $ \{A_i\}_{i=1}^N, A_i \in \bbR^{m\times n} $ is denoted by \rev{$ \{A_i\}_i^\diagbox $}.
Especially when $ A_i = A, \forall i $, $ \{A_i\}_i^\diagbox $ is also denoted by \rev{$  A^{\diagbox,N} $}.
Let $ (M,d) $ be a metric space.
The open ball of radius $ r > 0 $ centered at $ x \in M $ is denoted by $ B_r(x) := \{y\in M : d(x,y) < r \} $.
The element-wise division of $ a, b \in \bbR_{>0}^n $ is denoted by $ a \oslash b := [a_1 /b_1 \ \cdots \ a_n/b_n]^\top$.
The $ N $-dimensional vector of ones is denoted by $ 1_N $.
The gradient of a function $ f $ with respect to the variable $ x $ is denoted by $ \nabla_x f $.
\rev{For $ x, x' \in \bbR_{> 0}^n $, define an equivalence relation $ \sim $ on $ \bbR_{> 0}^n $ by $ x \sim x' $ if and only if $ \exists r> 0, x = r x' $.}

\section{Background on optimal transport}\label{sec:optimal_transport}
Here, we briefly review OT between discrete distributions $ \mu := \sum_{i=1}^{N} a_i \delta_{x_i}, \nu := \sum_{j=1}^{M} b_j \delta_{y_j} $ where $a \in \Sigma_N := \{ p\in \bbR_{\ge 0}^N : \sum_{i=1}^{N} p_i = 1 \}, b\in \Sigma_M $, $ x_i, y_j \in \bbR^n $, and $ \delta_x $ is the Dirac delta at $ x $.
Given a cost function $ c : \bbR^n \times \bbR^n (\ni (x,y)) \rightarrow \bbR $, which represents the cost of transporting a unit of mass from $ x $ to $ y $, the original formulation of OT due to Monge seeks a map $ T : \{x_1,\ldots,x_N\} \rightarrow \{y_1,\ldots,y_M\}$ that solves
\begin{equation}\label{prob:Monge}
	\begin{aligned}
	&\underset{T}{\rm minimize} \ \sum_{i \in \bbra{N}} c(x_i, T(x_i)) \\
	&\text{subject to} \ b_j = \sum_{i : T(x_i) = y_j} a_i, \ \forall j \in \bbra{M} .
	\end{aligned}
\end{equation}
Especially when $ M = N $ and $ a = b = 1_N /N $, the optimal map $ T $ gives the optimal assignment for transporting agents with the initial states $ \{x_i\}_i $ to the desired states $ \{y_j\}_j $, \red{and the Hungarian algorithm~\cite{Kuhn1955} can be used to solve \eqref{prob:Monge}.}
\red{However, this method can be applied only to small problems because it has $ O(N^3) $ complexity.}

On the other hand, the Kantorovich formulation of OT is a linear program (LP):
\begin{equation}\label{prob:Kantorovich}
	\underset{P\in \calT(a,b)}{\rm minimize} \ \sum_{i\in \bbra{N},j\in \bbra{M}} C_{ij} P_{ij}
\end{equation}
where $ C_{ij} := c(x_i, y_j) $ and
\[
	\calT(a,b) := \{ P \in \bbR_{\ge0}^{N\times M} : P 1_M = a, \ P^\top 1_N = b   \}.
\]
A matrix $ P \in \calT(a,b) $, which is called a coupling matrix, represents a transport plan where $ P_{ij} $ describes the amount of mass flowing from $ x_i $ towards $ y_j $.
In particular, when $ M=N $ and $ a = b = 1_N/N $, there exists an optimal solution from which we can reconstruct an optimal map for \eqref{prob:Monge} \cite[Proposition~2.1]{Peyre2019}.
\red{However, similarly to \eqref{prob:Monge}, for a large number of agents and destinations, the problem~\eqref{prob:Kantorovich} with $ N M $ variables is challenging to solve.}

In view of this, \cite{Cuturi2013} employed an entropic regularization to \eqref{prob:Kantorovich}:
\begin{equation}\label{prob:Sinkhorn}
		\underset{P\in \calT(a,b)}{\rm minimize} \ \sum_{i\in \bbra{N},j\in \bbra{M}} C_{ij} P_{ij} - \varepsilon H(P),
\end{equation}
where $ \varepsilon > 0 $ and the entropy of $ P $ is defined by $H(P) := - \sum_{i,j} P_{ij} (\log (P_{ij}) - 1)$. 
Define the Gibbs kernel $ K $ associated to the cost matrix $ C = (C_{ij}) $ as
\[
	K = (K_{ij}) \in \bbR_{> 0}^{N\times M}, \ K_{ij} := \exp\left(  - C_{ij}/\varepsilon \right) .
\]
Then, a unique solution of the entropic OT \eqref{prob:Sinkhorn} has the form $ P^* =  (\alpha^*)^\diagbox K  (\beta^*)^\diagbox $ for two (unknown) scaling variables $ (\alpha^*,\beta^*) \in \bbR_{>0}^N \times \bbR_{>0}^M $. The variables $ (\alpha^*, \beta^*) $ can be efficiently computed by the Sinkhorn algorithm:
\begin{equation}\label{eq:Sinkhorn_algorithm}
	\alpha(k+1) = a \oslash [K\beta(k)], \ \beta(k) = b \oslash [K^\top \alpha(k)]
\end{equation}
where
\[
	\lim_{k\rightarrow \infty} \alpha(k+1)^\diagbox K \beta  (k)^\diagbox  =  P^*, \ \forall \alpha(0) = \alpha_0 \in \bbR_{>0}^N. 
\]

\rev{Now, let us introduce Hilbert's projective metric}
\begin{equation}\label{eq:hilbert_metric}
	\dhil (\beta, \beta') := \log \max_{i,j\in \bbra{M}} \frac{\beta_i \beta'_j}{\beta_j \beta'_i}, \ \beta, \beta' \in \bbR_{>0}^M ,
\end{equation}
\rev{which is a distance on the projective cone $ \bbR_{>0}^M / {\sim} $ (see the Notation in Section~\ref{sec:intro} for $ \sim $)} and is useful for the convergence analysis of the Sinkhorn algorithm; see~\cite[Remark~4.12 \rev{and 4.14}]{Peyre2019}.
Indeed, for any $ (\beta, \beta') \in (\bbR_{> 0}^M )^2 $ and any $ \bar K \in \bbR_{>0}^{N\times M} $, it holds
\begin{equation}\label{ineq:hilbert_metric}
		\dhil (\bar K\beta, \bar K\beta') \le \lambda (\bar K) \dhil (\beta, \beta')
\end{equation}
where
\begin{equation*}
	\lambda (\bar K) := \frac{\sqrt{\eta(\bar K)} -1}{\sqrt{\eta(\bar K)} + 1} < 1,\ \eta(\bar K) := \max_{i,j,k,l} \frac{\bar{K}_{ik} \bar{K}_{jl}}{\bar{K}_{jk} \bar{K}_{il}} .
\end{equation*}
Then it follows from \eqref{ineq:hilbert_metric} that
\begin{align*}
	\dhil &(\beta(k+1), \beta^*) = \dhil (b\oslash [K^\top \alpha(k+1)], b\oslash [K^\top \alpha^*]) \\
	&=\dhil (K^\top \alpha(k+1), K^\top \alpha^*) \\
	&\le \lambda (K) \dhil (\alpha (k+1), \alpha^*) \le \lambda^2 (K) \dhil (\beta(k), \beta^*)
\end{align*}
which implies $ V_P (\beta) := \dhil (\beta, \beta^*) $ is a Lyapunov function of \eqref{eq:Sinkhorn_algorithm}, and $ \lim_{k\rightarrow \infty} \beta (k) = \beta^* \in \bbR_{>0}^M / {\sim} $.

\section{Problem formulation}\label{sec:formulation}
In this paper, we consider the problem of \rev{stabilizing agents efficiently to a given discrete distribution over dynamical systems.
This can be formulated as Monge's OT problem.
\begin{problem}
	\label{prob:main}
	Given initial and desired states $\{x_i^0\}_{i=1}^N, \{x_j^\sfd\}_{j=1}^N  \in (\mathbb{R}^{n})^N$,
	find control inputs $\{u_i\}_{i=1}^N$ and a permutation $ \sigma : \bbra{N} \rightarrow \bbra{N} $ that solve
	\begin{align}
		&\underset{\sigma }{\rm minimize} ~~ \sum_{i\in \bbra{N}} c_\infty^i (x_i^0, x_{\sigma(i)}^\sfd ) .  \label{eq:infinite_horizon_cost}
	\end{align}
	Here, the cost function $ c_\infty^i $ is defined by
	\begin{align}
		& &&\hspace{-1cm} c_\infty^i (x_i^0, x_j^\sfd) := \min_{u_i} \ \sum_{k=0}^\infty  \ell_i (x_i (k), u_i (k) ; x_j^\sfd) \label{eq:value_infty}\\
		&\hspace{-0.3cm} \text{subject to} &&\red{x_i (k+1) = A_i x_i (k) + B_i u_i (k)}, \label{eq:nonlinear_dynamics}\\
		& &&x_i (k) \in \calX \subseteq \bbR^n , \ \forall k \in \bbZ_{\ge 0}, \label{eq:state_constraint} \\ 
		& &&x_i(0) = x_{i}^0, \label{eq:constraint_initial} \\
		& &&\lim_{k\rightarrow \infty} x_i(k) =  x_j^\sfd,\label{eq:constraint_infty}
	\end{align}
	where $ x_i(k) \in \bbR^n, u_i(k) \in \bbR^m, A_i\in \bbR^{n\times n}, B_i \in \bbR^{n\times m} $.
	\hfill $ \triangle $
\end{problem}
}
Note that the running cost $ \ell_i $ depends not only on the state $ x_i $ and the control input $ u_i $, but also on the destination $ x_{j}^\sfd $. \rev{Throughout this paper, we assume the existence of an optimal solution of OC problems. \red{In addition, we assume that there exists a constant input $ \bar{u}_i $ under which $ x_i = x_j^\sfd $ is an equilibrium of \eqref{eq:nonlinear_dynamics}.}
\red{A necessary condition for the infinite horizon cost $ c_\infty^i (x_i^0, x_j^\sfd) $ to be finite is that at $ x_i = x_j^\sfd $ with $ u_i = \bar{u}_i $, there is not a cost incurred, i.e., $ \ell_i(x_j^\sfd, \bar{u}_i ; x_j^\sfd) = 0 $.}
If $ B_i $ is square and invertible, $ \bar{u}_i = B_i^{-1} ( x_j^\sfd - A_ix_j^\sfd) $ makes $ x_i = x_j^\sfd $ an equilibrium.
}

\section{Combining MPC and Sinkhorn algorithm}\label{sec:Sinkhorn_MPC}
The main difficulties of Problem~\ref{prob:main} are as follows:
\begin{enumerate}
	\item In most cases, the infinite horizon OC problem $ c_\infty^{\magenta{i}} (x_i^0,x_j^\sfd) $ is computationally intractable.
	\item \red{In addition, given $ c_\infty^{\magenta{i}} (x_i^0, x_j^\sfd), \forall i,j\in \bbra{N} $, the assignment problem needs to be solved, which leads to the high computational burden when $ N $ is large.}
\end{enumerate}
To overcome these issues, we utilize the concept of MPC, \magenta{which solves tractable finite horizon OC while satisfying the state constraint~\eqref{eq:state_constraint}.}
Specifically, at each time, we address the OT problem whose cost function with the current state $ \check{x}_i $ and the finite horizon $ T_h \in \bbZ_{>0} $ is given by
\begin{align*}
	c_{T_h}^{\magenta{i}} (\check{x}_i, x_j^\sfd) := &\min_{u_i} \ \sum_{k=0}^{T_h-1} \ell_i (x_i(k), u_i (k); x_j^\sfd) \\
	&\text{subj. to  \eqref{eq:nonlinear_dynamics}, \magenta{\eqref{eq:state_constraint}},}~~ x_i(0) = \check{x}_i, \  x_i (T_h) = x_{j}^\sfd .
\end{align*}
Denote the first control in the optimal sequence by $ u_i^{\rm MPC} (\check{x}_i, x_j^\sfd) $.
From the viewpoint of the computation time for solving the OT problem, we relax the Problem~\ref{prob:main} by the entropic regularization.
It should be noted that, in challenging situations in which the number of agents is large and the sampling time is small, only a few Sinkhorn iterations are allowed.
In what follows, we deal only with the case where {\em just one} Sinkhorn iteration is performed at each time. Nevertheless, by similar \rev{arguments}, all \rev{of} the results of this paper are still valid when more iterations are performed.

Based on the approximate solution $ P(k) $ obtained by the Sinkhorn algorithm at time $ k $, we \rev{need} to determine a target state for each agent.
Then, we introduce a set $ \bbX \subset \bbR^n $ and a map $ x_{\rm tmp}^{\sfd,i} : \bbR_{\ge 0}^{N\times N} \rightarrow \bbX $ as a policy to determine a temporary target $ x_{\rm tmp}^{\sfd,i} (P(k)) $ at time $ k $ for agent $ i $.
Hereafter, assume that there exists a constant $ r_{\rm upp} > 0 $ such that
\begin{equation}\label{ineq:x_upp}
	\|x \| \le  r_{\rm upp} , \ \forall x \in \bbX .
\end{equation}
For example, if $ \bbX $ is the convex hull of $ \{x_j^\sfd\}_j $, we can take $ r_{\rm upp} = \max_j \|x_j^\sfd\| $. A typical policy to approximate a Monge's OT map from a coupling matrix $ P $ is the so-called barycentric projection $ x_{\rm tmp}^{\sfd,i} (P) = N \sum_{j=1}^{N} P_{ij} x_j^\sfd $~\cite[Remark~4.11]{Peyre2019}.

\rev{For any given policy $ x_{\rm tmp}^{\sfd, i} $,} we summarize the above strategy as the following dynamics where the Sinkhorn algorithm behaves as a dynamic controller.\\
{\bf Sinkhorn MPC:}
\begin{align}
&x_i(k+1) = A_i x_i (k) + B_i u_i^{\rm MPC} \bigl(x_i (k), x_{\rm tmp}^{\sfd,i} \left(P(k) \right)  \bigr),\nonumber\\
&\hspace{6cm} \forall i \in \bbra{N}, \label{eq:SMPC_x} \\
&P(k) = \alpha(k+1)^\diagbox K(x(k)) \beta(k)^\diagbox, \label{eq:SMPC_P} \\
&\alpha(k+1) = 1_N/N \oslash \left[K(x(k))\beta(k) \right], \label{eq:SMPC_alpha} \\
&\beta(k) = 1_N/N \oslash \left[K(x(k))^\top \alpha(k) \right], \label{eq:SMPC_beta}   \\
&x_i (0) = x_i^0, \ \alpha(0) = \alpha_0, \nonumber 
\end{align}
where
\begin{align*}
K_{ij}(x) := \exp\left(- \frac{c_{T_h}^{\magenta{i}} (x_i, x_{j}^\sfd )}{\varepsilon} \right), \ x = [x_1;  \cdots  ;x_N] \in \bbR^{nN} ,
\end{align*}
and the initial value $ \alpha_0 $ is arbitrary.
\hfill $ \triangle $
\\
\noindent
Note that, \rev{under the assumption that for all $ i\in \bbra{N} $, $ B_i $ is square and invertible, Sinkhorn MPC is obviously recursively feasible, and} the dynamics~\eqref{eq:SMPC_x} satisfies the state constraint~\eqref{eq:state_constraint} for all $ i\in \bbra{N} $.

\section{Numerical examples}\label{sec:example}
This section gives examples for Sinkhorn MPC \modi{with an energy cost
\begin{equation}\label{eq:energy}
	\ell_i (x_i, u_i; x_j^\sfd) = \|u_i - B_i^{-1} (x_j^\sfd - A_i x_j^\sfd)\|^2 ,
\end{equation}
and $ \calX = \bbR^n $.}
\modi{Then, the dynamics under Sinkhorn MPC can be written as follows \cite[\rev{Section~2.2, pp.~37-39}]{Lewis2012}:
\begin{align}
&x_i(k+1) = \bar{A}_i x_i (k) +  (I - \bar{A}_i) x_{\rm tmp}^{\sfd,i}(P(k)), \label{eq:SMPC_linear_x} \\
&u_i^{\rm MPC} (x_i, \red{\hat{x}}) = - B_i^\top (A_i^\top)^{T_h - 1} G_{i,T_h}^{-1} A_i^{T_h} (x_i - \red{\hat{x}} ) \nonumber\\
&\hspace{2.5cm} + B_i^{-1} ( \red{\hat{x}} - A_i \red{\hat{x}}), \ \forall i \in \bbra{N}, \ \red{\forall \hat{x}\in \bbR^n} \nonumber
\end{align}
with \eqref{eq:SMPC_P}--\eqref{eq:SMPC_beta} where
\begin{align*}
&K_{ij}(x)  = \exp \left(- \frac{  \|x_i - x_j^\sfd\|_{\calG_i}^2 }{\varepsilon}   \right), \\
&\calG_i := (A_i^{T_h})^\top G_{i,T_h}^{-1}  A_i^{T_ h}, \ G_{i,T_h} := \sum_{k=0}^{T_h -1} A_i^k B_i B_i^\top (A_i^\top)^k, \\
&\bar{A}_i := A_i - B_i B_i^\top (A_i^\top)^{T_h -1} G_{i,T_h}^{-1} A_i^{T_h} .
\end{align*}
In the examples below, we use the barycentric target $ x_{\rm tmp}^{\sfd,i} ( P) = N \sum_{j=1}^{N} P_{ij} x_j^\sfd $.}

First, consider the case where
\begin{equation}\label{eq:ex2d}
	A_i = 
	\begin{bmatrix}
		1.2 & 0.13\\
		-0.05 & 1.1
	\end{bmatrix}, \
	B_i = 0.1 I_2, \ \forall i\in \bbra{N},
\end{equation}
and set $ N = 150, \ \varepsilon = 1.0, \ T_h = 10, \ \alpha_0 = 1_N $.
For given initial and desired states, the trajectories of the agents governed by \eqref{eq:SMPC_linear_x} with \eqref{eq:SMPC_P}--\eqref{eq:SMPC_beta} are illustrated in Fig.~\ref{fig:trajectory}.
It can be seen that the agents converge sufficiently close to the target states.
The computation time for one Sinkhorn iteration is about $0.0063$~ms, $0.030$~ms, and $ 0.08 $~ms for $ N = 150, 500, 800 $, respectively, with MacBook~Pro with Intel~Core~i5. 
On the other hand, solving the linear program \eqref{prob:Kantorovich} with MATLAB {\tt linprog} takes about $0.12$~s, $6.4$~s, and $ 66 $~s for $ N = 150, 500, 800 $, respectively, and is thus not scalable.
Hence, Sinkhorn MPC contributes to reducing the computational burden.

\begin{figure}[t]
	\begin{minipage}[b]{0.48\linewidth}
		\centering
		\includegraphics[keepaspectratio, scale=0.24]
		{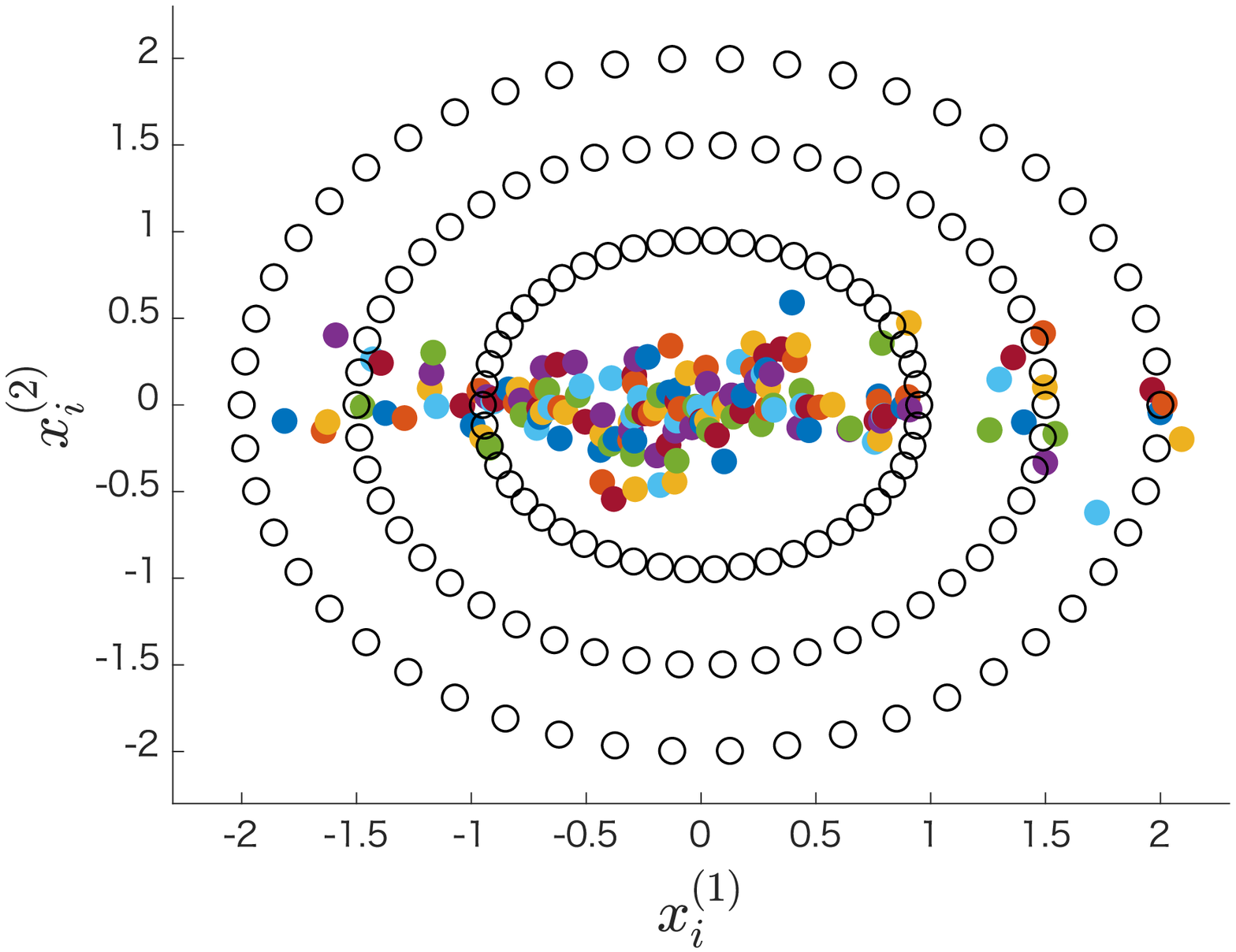}
		\subcaption{$ k=0 $}
	\end{minipage}
	\vspace{0.5cm}
	\begin{minipage}[b]{0.5\linewidth}
		\centering
		\includegraphics[keepaspectratio, scale=0.23]
		{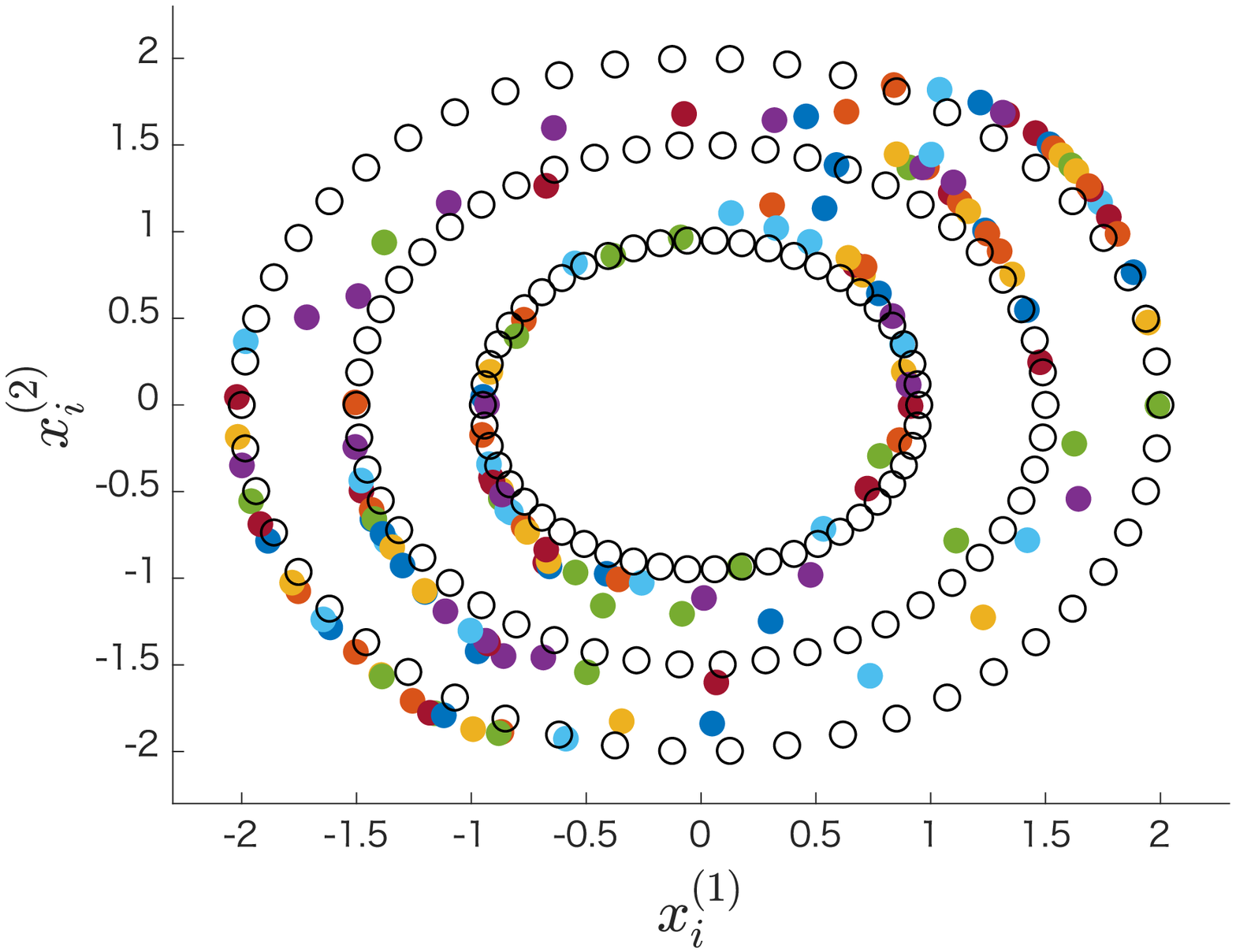}
		\subcaption{$ k=50 $}
	\end{minipage}
	\vspace{0.5cm}
	\begin{minipage}[b]{0.48\linewidth}
		\centering
		\includegraphics[keepaspectratio, scale=0.23]
		{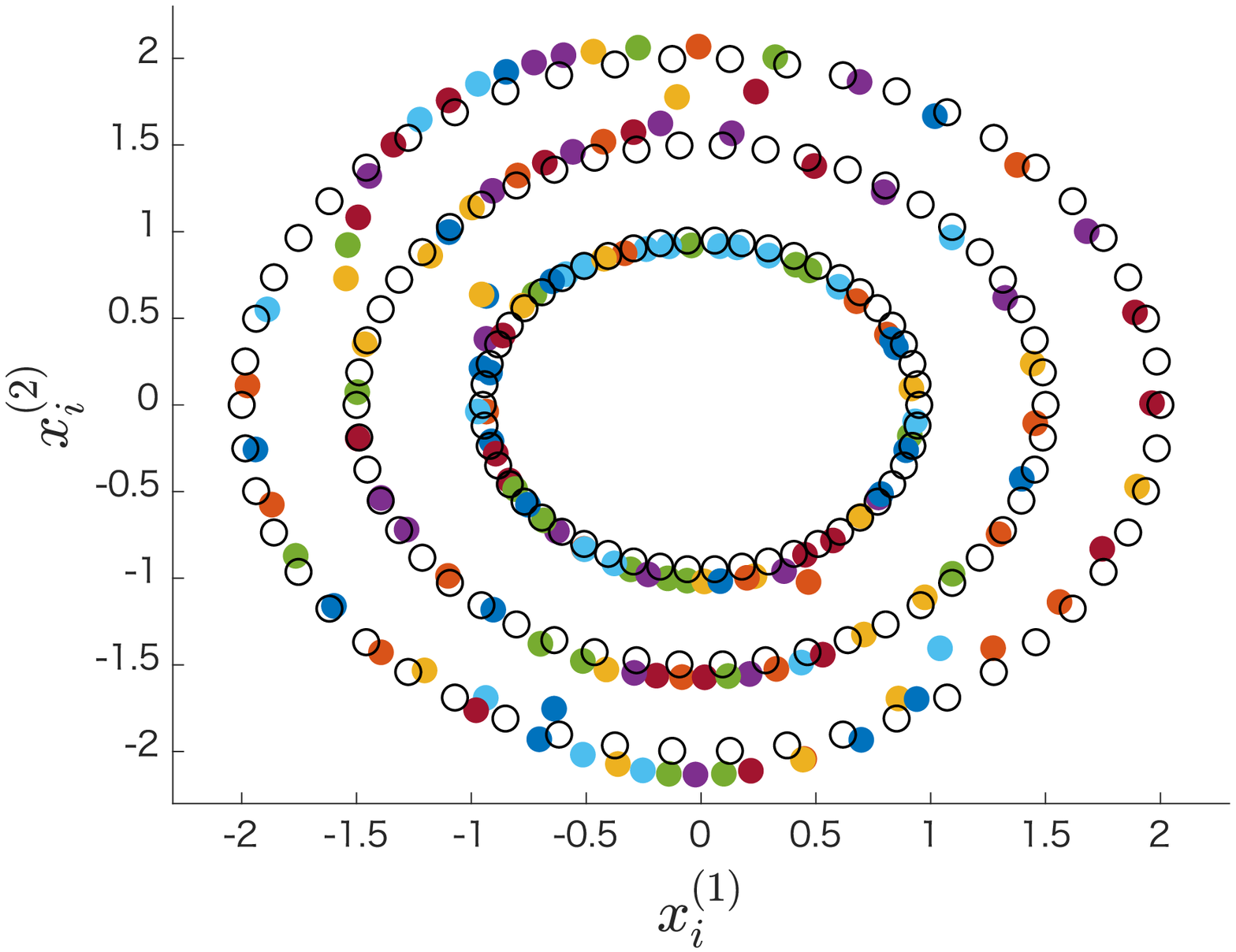}
		\subcaption{$ k=200 $}
	\end{minipage}
	\begin{minipage}[b]{0.5\linewidth}
		\centering
		\includegraphics[keepaspectratio, scale=0.23]
		{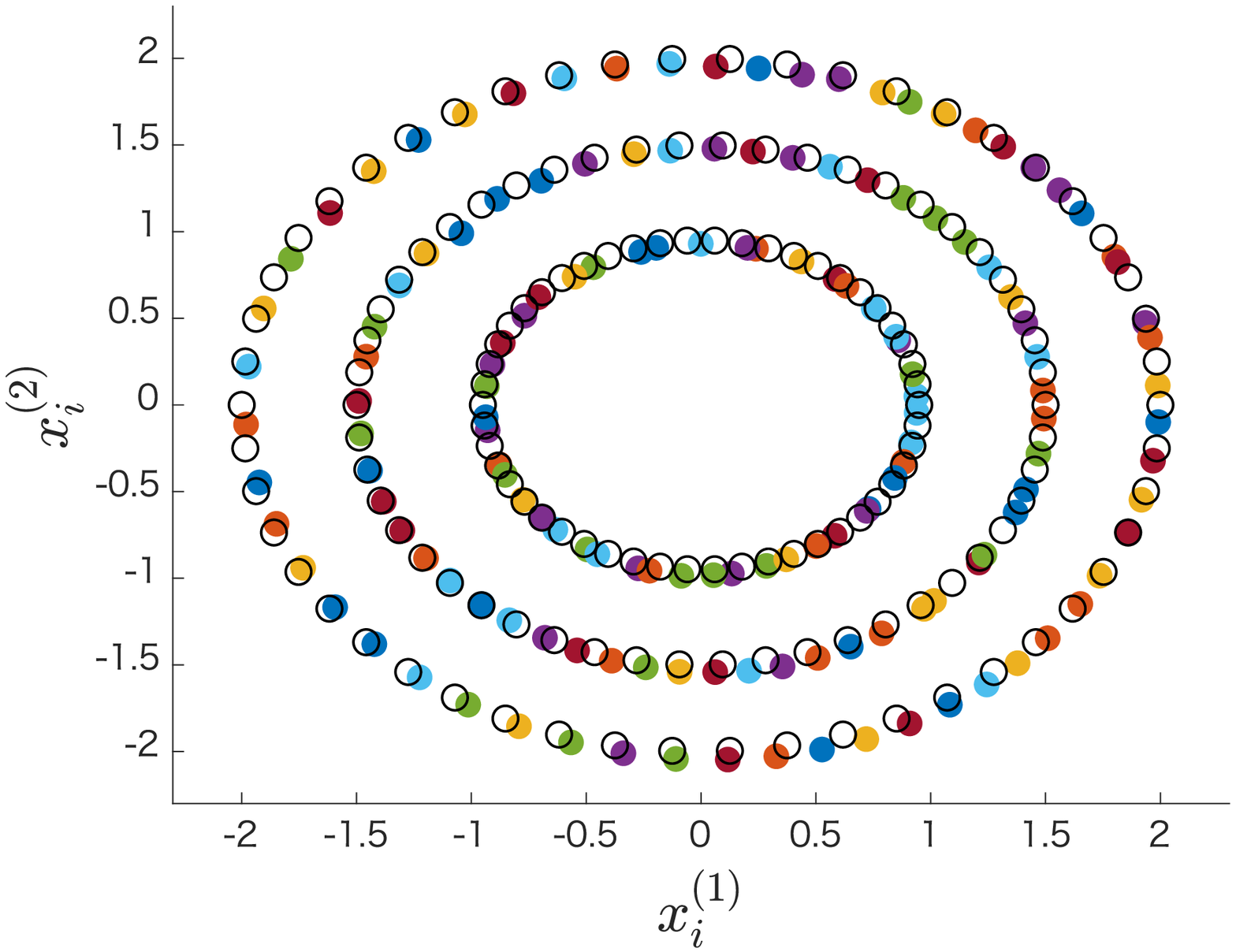}
		\subcaption{$ k=500 $}
	\end{minipage}
	
	\begin{minipage}[b]{0.48\linewidth}
		\centering
		\includegraphics[keepaspectratio, scale=0.23]
		{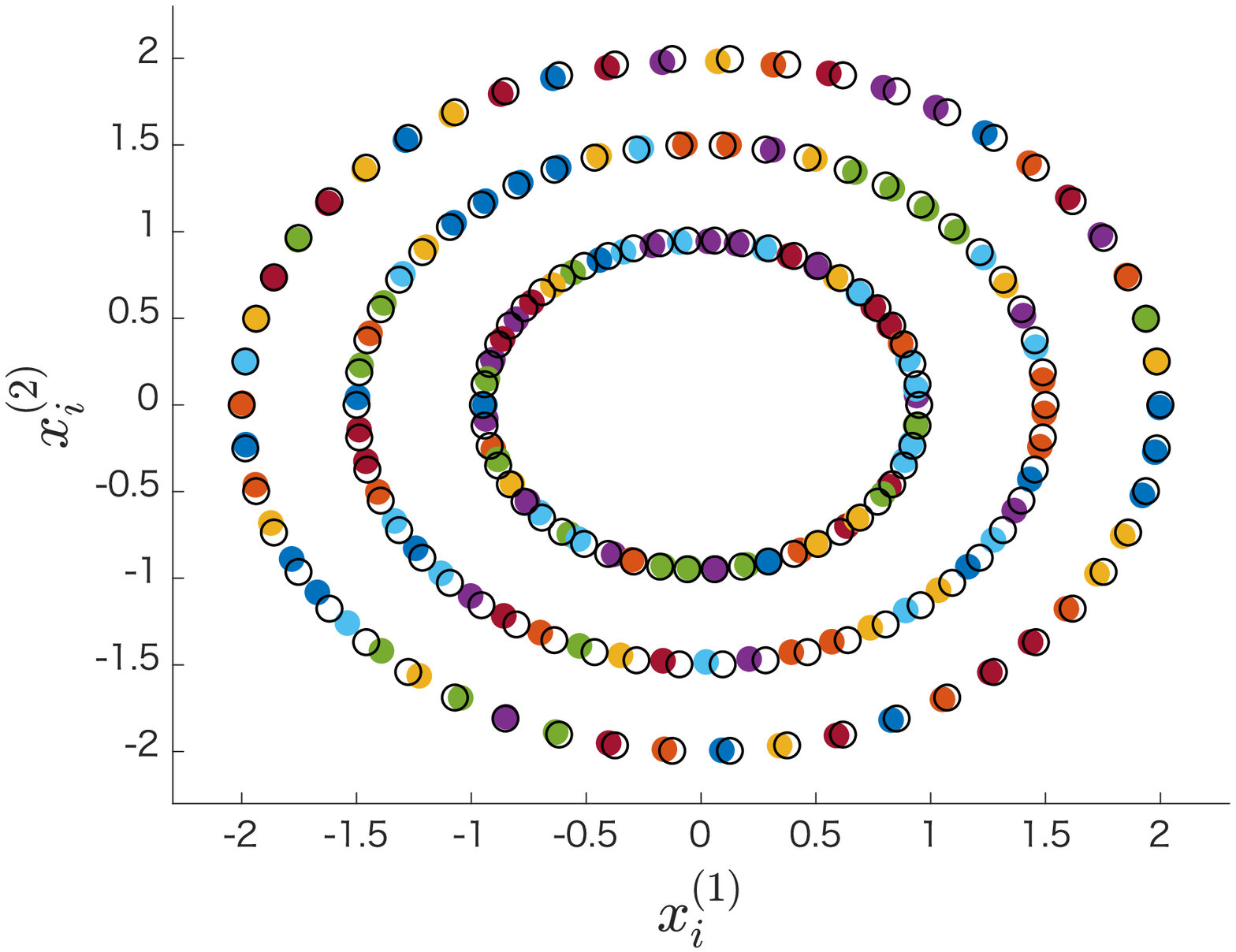}
		\subcaption{$ k=1000 $}
	\end{minipage}
	\begin{minipage}[b]{0.5\linewidth}
		\centering
		\includegraphics[keepaspectratio, scale=0.23]
		{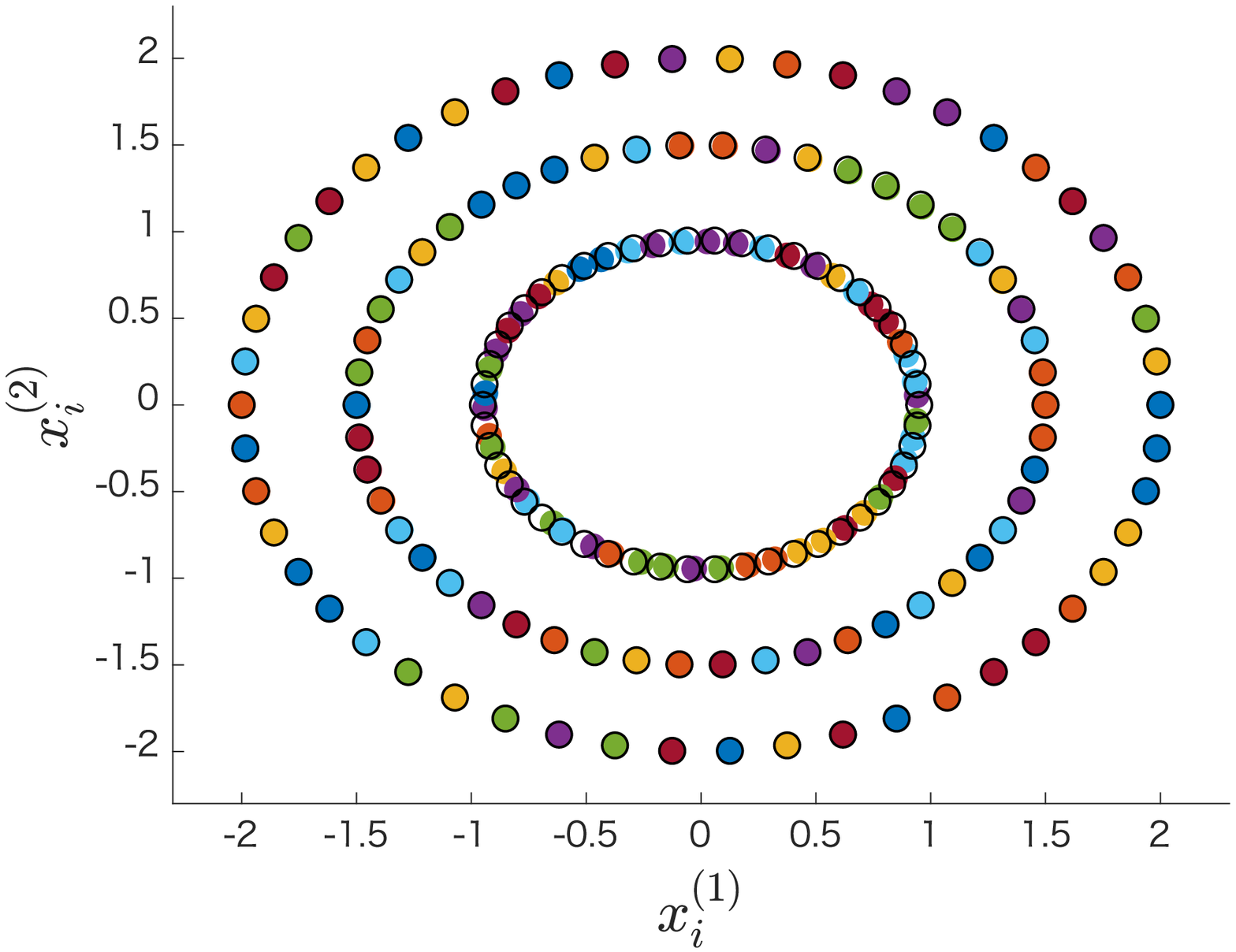}
		\subcaption{$ k=3000 $}
	\end{minipage}
	\caption{Trajectories $ x_i(k) = [x_i^{(1)}(k) \ x_i^{(2)} (k)]^\top $ of 150 agents for \eqref{eq:ex2d} (colored filled circles) and desired states (black circles).}\label{fig:trajectory}
\end{figure}

Next, we investigate the effect of the regularization parameter $ \varepsilon $ on the behavior of Sinkhorn MPC. To this end, consider a simple case where $ N = 10, \ T_h = 10 $, and
\begin{equation}\label{eq:ex1d}
	A_i = 1, \ B_i = 0.1, \ \forall i \in \bbra{N}.
\end{equation}
Then the trajectories of the agents with $ \varepsilon = 0.5, 1.0 $ are shown in Fig.~\ref{fig:ex1d}.
As can be seen, the overshoot/undershoot is reduced for larger $ \varepsilon $ while the limiting values of the states deviate from the desired states.
In other words, the parameter $ \varepsilon $ reflects the trade-off between the stationary and transient behaviors of the dynamics under Sinkhorn MPC.

\begin{figure}[tb]
	\centering
	\includegraphics[scale=0.35]{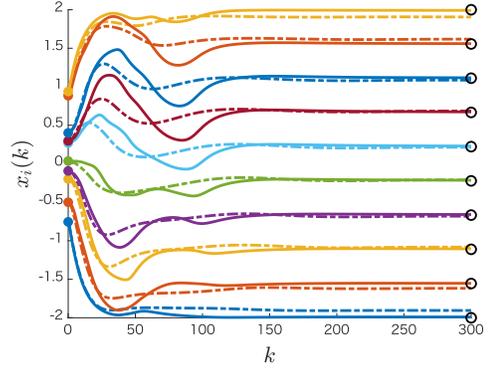}
	\caption{Trajectories of 10 agents for \eqref{eq:ex1d} with $ \varepsilon = 0.5 $ (solid) and $ \varepsilon = 1.0 $ (chain), respectively, and desired states (black circles).}
	\label{fig:ex1d}
\end{figure}

\section{Fundamental properties of Sinkhorn MPC with an energy cost}\label{sec:linear_system}
In this section, we elucidate ultimate boundedness and asymptotic stability of Sinkhorn MPC with \modi{the energy cost \eqref{eq:energy}} and $ \calX = \bbR^n $. \rev{Hereafter, we assume the invertibility of $ B_i $.}

\subsection{Ultimate boundedness for Sinkhorn MPC}\label{subsec:bounded}
It is known that, under the assumption that $ B_i $ is invertible, $ \bar{A}_i $ is stable, i.e., \blue{the spectral radius $ \rho_i $ of $ \bar{A}_i $ satisfies $ \rho_i < 1 $} \cite[Corollary~1]{Kwon1975}.
Using this fact, we derive the ultimate boundedness of \eqref{eq:SMPC_linear_x} with \eqref{eq:SMPC_P}--\eqref{eq:SMPC_beta}.

\blue{
\begin{proposition}\label{prop:bounded}
	For any $ \delta > 0, \{x_i^0\}_i $, and $ \{\nu_i\}_i $ satisfying $ \nu_i > 0, \rho_i + \nu_i < 1, \forall i $, there exists $ \tau(\delta, \{x_i^0\},\{\nu_i \} ) \in \bbZ_{>0} $ such that the solution $ \{x_i\}_i $ of \eqref{eq:SMPC_linear_x} with \eqref{eq:SMPC_P}--\eqref{eq:SMPC_beta} satisfies
	\begin{equation}\label{ineq:ultimate_bound}
		\|x_i (k) \| < \delta  +  \frac{ r_{\rm upp} \|I- \bar{A}_i\|_2}{1 - (\rho_i + \nu_i ) }  , \ \forall k \ge \blue{\tau}, \ \forall i \in \bbra{N},
	\end{equation}
	where $ r_{\rm upp} $ satisfies \eqref{ineq:x_upp}.
\end{proposition}
}
\begin{proof}
	Let $ \tilde{u}_i (k) := ( I - \bar{A}_i) x_{\rm tmp}^{\sfd,i} (P(k)) $. 
	Then, it follows from \eqref{ineq:x_upp} that
	\begin{align*}
	\| \tilde{u}_i (k) \| \le r_{\rm upp} \| I -  \bar{A}_i \|_2, \ \forall k\in \bbZ_{\ge 0} .
	\end{align*}
	\blue{Note that for any $ \nu_i > 0$, there exists $ \tau_i (\nu_i) \in \bbZ_{>0} $ such that}
	\[
		\blue{ \| \bar{A}_i^k \| < (\rho_i + \nu_i )^k, \ \forall k \ge \tau_i . }
	\]
	Hence, the desired result is straightforward from
	\begin{align*}
	\| x_i (k) \| \le \| \bar{A}_i^k \|_2 \| x_i^0 \| + \sum_{s = 1}^{k} \| \bar{A}_i^{s-1} \|_2 \| \tilde{u}_i (k-s)\| .
	\end{align*}
\end{proof}
\rev{We emphasize that Proposition~\ref{prop:bounded} holds for any policy $ x_{\rm tmp}^{\sfd, i} $ \fin{whose range $ \bbX $ satisfies \eqref{ineq:x_upp}.}}

\subsection{Existence of the equilibrium points}\label{subsec:equilibrium}
In the \rev{remainder} of this section, we focus on the barycentric target $ x_{\rm tmp}^{\sfd,i} ( P) = N \sum_{j=1}^{N} P_{ij} x_j^\sfd $. 
For $ (x,\beta) \in \bbR^{nN} \times \bbR_{>0}^N $ and $ X^\sfd := [x_1^\sfd \ \cdots \ x_N^\sfd] \in \bbR^{n\times N} $, define
\begin{align}
&f_1 (x,\beta) := \{\bar{A}_i\}_i^\diagbox x + N  \{I_n - \bar{A}_i\}_i^\diagbox  (X^\sfd)^{\diagbox,N} {\rm vec}(\tilde{P}(x,\beta)), \\
&f_2 (x,\beta) := 1_N/N \oslash \left[ K(f_1 (x,\beta))^\top (1_N/N \oslash [K(x) \beta]) \right], \\
&\tilde{P}(x,\beta) :=  \left( 1_N/N \oslash [K(x)\beta]  \right)^\diagbox K(x)  \beta^\diagbox. 
\end{align}
Then, the collective dynamics \eqref{eq:SMPC_linear_x} with \eqref{eq:SMPC_P}--\eqref{eq:SMPC_beta} is
\begin{align}
&x(k+1) = f_1 (x(k), \beta(k))  , \label{eq:x_global}\\
&\beta(k+1) = f_2(x(k), \beta(k)) , \label{eq:beta_global}
\end{align}
where $ x(k) := \rev{[x_1 (k); \ \cdots ; x_N (k)]} $.

Here, we characterize equilibria of \eqref{eq:x_global}, \eqref{eq:beta_global}. Note that the existence of the equilibria is not trivial due to the nonlinearity of the dynamics. 
A point $ x^{\sfe} = \rev{[x_1^\sfe; \ \cdots \ ;x_N^\sfe ]} \in \bbR^{nN} $ is an equilibrium if and only if
\begin{align}
&\hspace{-0.2cm}(I_n - \bar{A}_i) \biggl( x_{i}^{\sfe} - N \sum_{j=1}^{N} P_{ij}^* (x^\sfe) x_j^\sfd \biggr) = 0 , \ \forall i \in \bbra{N}, \nonumber\\
&\hspace{-0.2cm}P_{ij}^* (x) :=  \alpha_{i}^* K_{ij}(x) \beta_{j}^*,\ \alpha^*, \beta^* \in \bbR_{>0}^N,\\
&\hspace{-0.2cm}\alpha^* = 1_N/N \oslash \left[ K(x) \beta^*	\right],  \beta^* = 1_N/N \oslash \left[ K(x)^\top \alpha^*	\right]. \label{eq:optimal_beta}
\end{align}
The stability of $ \bar{A}_i $ implies that it has no eigenvalue \rev{equal to} $ 1 $, and therefore $ I_n - \bar{A}_i $ is invertible. Thus, the necessary and sufficient condition for the equilibria is given by
\begin{equation}\label{eq:equilibrium_iff}
x_{i}^{\sfe} - N\sum_{j=1}^{N} P_{ij}^* (x^\sfe) x_j^{\sfd} = 0, \ \forall i\in \bbra{N} .
\end{equation}

\begin{proposition}\label{prop:fixed_point}
	The dynamics \eqref{eq:x_global}, \eqref{eq:beta_global} has at least one equilibrium point $ (x^\sfe, \beta^\sfe) \in \bbR^{nN} \times (\bbR_{> 0}^N / {\sim}) $.
\end{proposition}
\begin{proof}
	Note that if a point $ x^{\sfe}\in \bbR^{nN} $ satisfies \eqref{eq:equilibrium_iff}, the corresponding $ \beta^\sfe \in \bbR_{> 0}^N / {\sim} $ is uniquely determined by $ \beta^\sfe = \beta^* $ in \eqref{eq:optimal_beta} with $ x= x^\sfe $~\cite[Theorem~4.2]{Peyre2019}.
	Define a continuous map $ h : \bbR^{nN} \rightarrow \bbR^{nN} $ as
	\begin{equation}
	h(x) := 
	\begin{bmatrix}
	  N\sum_{j=1}^N P_{1j}^* (x) x_j^\sfd \\
	\vdots \\
	  N\sum_{j=1}^N P_{Nj}^* (x) x_j^\sfd
	\end{bmatrix}
	, \ x \in \bbR^{nN} . \label{eq:f_equilibrium}
	\end{equation}
	From \eqref{eq:equilibrium_iff}, fixed points of $ h $ are equilibria of \eqref{eq:x_global}, \eqref{eq:beta_global}.
	Note that for any $ i \in \bbra{N}$ and any $ x \in \bbR^{nN} $, $ N\sum_{j=1}^N P_{ij}^* (x) x_j^\sfd $ belongs to the convex hull $ \bbX $ of $ \{x_j^\sfd\}_j $.
	For brevity, we abuse notation and regard $ \bbX^N $ as a subset of $ \bbR^{nN} $.
	Let $ h_{\bbX} : \bbX^N \rightarrow \bbX^N $ be the restriction of $ h $ in \eqref{eq:f_equilibrium} to $ \bbX^N $.
	Now we can utilize Brouwer's fixed point theorem. That is, since $ h_{\bbX} $ is a continuous map from a compact convex set $ \bbX^N $ into itself, there exists a point $ x^{\sfe} \in \bbX^N $ such that $ x^{\sfe} = h(x^\sfe) $.
\end{proof}

Sometimes, in order to emphasize the dependence of $ (x^\sfe, \beta^\sfe) $ on $ \varepsilon $, we write $ (x^\sfe (\varepsilon), \beta^\sfe (\varepsilon)) $.

\subsection{Asymptotic stability for Sinkhorn MPC}
Next, we analyze the stability of the equilibrium points.
For this purpose, the following lemma is crucial when $ \varepsilon $ is small.
\rev{Due to the limited space, we omit the proof.}
\begin{lemma}\label{lemma:exponential_convergence}
	Assume that $ x_i^\sfd \neq x_j^\sfd $ for all $(i,j), \ i\neq j $.
	For a permutation $ \sigma : \bbra{N} \rightarrow \bbra{N} $, define $ x^{\sfd} (\sigma) := \rev{[x_{\sigma(1)}^{\sfd}; \ \cdots \ ;x_{\sigma(N)}^{\sfd} ]} $ and a permutation matrix $ P^\sigma = (P_{ij}^\sigma)$ as $ P_{ij}^\sigma := 1/N $ if $ j= \sigma(i) $, and $ 0 $, otherwise.
	Then, there exists an equilibrium $ (x^\sfe(\varepsilon), \beta^\sfe(\varepsilon)) $ of \eqref{eq:x_global}, \eqref{eq:beta_global} such that $ x^\sfe (\varepsilon) $ and $ P^* (x^\sfe (\varepsilon)) $ converge exponentially to $ x^{\sfd} (\sigma) $ and $ P^\sigma $, respectively, as $ \varepsilon \rightarrow +0 $, i.e., there exists $ \zeta > 0 $ such that
	\[
		\lim_{\varepsilon \rightarrow +0} \frac{\|\eta(\varepsilon)  \|_2}{\exp (-\zeta/\varepsilon)} = 0
	\]
	for $ \eta (\varepsilon) = x^\sfe (\varepsilon) - x^\sfd (\sigma) $ and $ \eta (\varepsilon) = P^*(x^\sfe (\varepsilon)) - P^\sigma$.
	\hfill $ \triangle $
\end{lemma}
Denote by $ {\rm Exp}(\sigma) $ the set of all equilibria \fin{$ (x^\sfe (\cdot), \beta^\sfe (\cdot)) $} of \eqref{eq:x_global}, \eqref{eq:beta_global} having the property in Lemma~\ref{lemma:exponential_convergence} for a permutation $ \sigma $.

For $ \bar{P} \in \bbR^{N\times N} $ and $ x  = \rev{[x_1; \ \cdots \ ;x_N]} \in \bbR^{nN} $, define
\begin{align*}
	V_{\rm x}(x) := \sum_{i=1}^{N} \Bigl\| x_i - N \sum_{j=1}^{N} \bar{P}_{ ij} x_j^\sfd \Bigr\|_{\calG_i}^2 .  
\end{align*}
Then, $ V_{\rm x} $ is a Lyapunov function of \eqref{eq:x_global} where $ \tilde{P}(x,\beta) $ is fixed by $ \bar{P} $~\cite{Mayne2000}.
Indeed, we have
\begin{align*}
	&V_{\rm x} (x(k+1)) - V_{\rm x} (x(k)) \le - \sum_{i=1}^N W_{1,i} (x_i(k), \bar{P}) \\
	&W_{1,i} (x_i, \bar{P}) := \Bigl\| B_i^\top (A_i^\top)^{T_h -1} G_{i,T_h}^{-1} A_i^{T_h} \\
	&\hspace{4cm} \times \bigl(x_i - N\sum_{j=1}^N \bar{P}_{ij} x_j^\sfd \bigr) \Bigr\|^2 . 
\end{align*}
Given an equilibrium $ (x^\sfe, \beta^\sfe) $, let us take the optimal coupling $ P^\sfe := P^* (x^{\sfe}) $ as $ \bar{P} $, and for $ \gamma > 0 $, define
\begin{equation}
V(x, \beta) :=  V_{\rm x} (x) +  \gamma \dhil ( \beta, \beta^\sfe ), \ (x,\beta)\in \bbR^{nN} \times  (\bbR_{>0}^N / {\sim} ) . 
\end{equation}
The following theorem follows from the fact that, for sufficiently small or large $ \varepsilon > 0$ and large $ \gamma > 0 $, $ V $ behaves as a Lyapunov function of \eqref{eq:x_global},~\eqref{eq:beta_global} with respect to $ (x^\sfe, \beta^\sfe) $.

\begin{theorem}
	Assume that for all $ i\in \bbra{N} $, $ A_i $ is invertible.
	Then the following hold:
	\begin{itemize}
		\item[(i)] Assume that $ (x^\sfe, \beta^\sfe) $ is an isolated equilibrium\footnote{An equilibrium is said to be isolated if it has a neighborhood which does not contain any other equilibria.} of \eqref{eq:x_global}, \eqref{eq:beta_global}. Then, for a sufficiently large $ \varepsilon > 0 $, $ (x^\sfe, \beta^\sfe) $ is locally asymptotically stable.
		\item[(ii)]  
		Assume that $ x_i^\sfd \neq x_j^\sfd $ for all $(i,j), \ i\neq j $. \fin{Assume further that for some $ \varepsilon' > 0 $, $ (x^\sfe (\varepsilon'), \beta^\sfe(\varepsilon')) $ is an isolated equilibrium of \eqref{eq:x_global}, \eqref{eq:beta_global} and for some permutation $ \sigma $, $ (x^\sfe (\cdot), \beta^\sfe(\cdot)) \in {\rm Exp}(\sigma) $.}
		Then, for sufficiently small $ \varepsilon >0$, $ (x^\sfe (\varepsilon), \beta^\sfe(\varepsilon)) $ is locally asymptotically stable.
	\end{itemize}
\end{theorem}
\begin{proof}
	We prove only (ii) as the proof is similar for (i).
	In this proof, we regard $ (x (\cdot), \beta(\cdot) ) $ as a trajectory in a metric space $ \bbR^{nN} \times (\bbR_{>0}^N / {\sim} ) $ with metric $ d((x,\beta), (x',\beta')) := \|x - x'\| + \dhil (\beta, \beta') $.
	Fix any $ (x^\sfe, \beta^\sfe) \in {\rm Exp}(\sigma) $ satisfying the assumption in (ii).
	By definition, it is trivial that $ V $ is positive definite on a neighborhood of $ (x^\sfe, \beta^\sfe) $.
	Moreover, for any $ (x,\beta) \in \bbR^{nN} \times (\bbR_{>0}^N / {\sim})$, we have
	\begin{align*}
	&V(f_1 (x, \beta) , f_2 (x, \beta) ) - 	V(x , \beta ) \\
	&\le \sum_{i=1}^{N} \biggl\{ \Bigl\| \bar{A}_i \Bigl(x_i  -  N \sum_j \tilde{P}_{ij}(x,\beta)  x_j^{\sfd} \Bigr) \\
	&+ N\sum_j (\tilde{P}_{ij}(x,\beta) - P_{ ij}^\sfe  )x_j^{\sfd}  \Bigr\|_{\calG_i}^2 -  \Bigl\|x_i  - N\sum_j P_{ ij}^\sfe x_j^\sfd \Bigr\|_{\calG_i}^2 \biggr\}  \\
	&\quad + \gamma ( - W_3 (x,\beta) + W_4(x,\beta) + W_5 (x,\beta) ) \\
	&\le \sum_{i=1}^{N} \left(- W_{1,i}(x_i, \tilde{P}(x,\beta)) + W_{2,i}(x ,\beta) \right) \\
	&\qquad + \gamma ( - W_{3}(x,\beta) + W_{4}(x,\beta) + W_{5} (x,\beta)) =: W(x,\beta),
	\end{align*}
	where we used the triangle inequality for $ \dhil $, and
	\begin{align*}
		&W_{2,i} (x,\beta) := 2\Bigl( x_i  - N\sum_{j\in \bbra{N}} \tilde{P}_{ij}(x,\beta) x_j^\sfd \Bigr)^\top (\bar{A}_i - I_n )^\top \calG_i  \\
		&\qquad\qquad\qquad \times N\sum_{j\in \bbra{N}} ( \tilde{P}_{ij}(x,\beta) - P_{ij}^\sfe) x_j^\sfd , \\
		&W_{3} (x,\beta) := \left[ 1- \lambda (K(x)) \lambda \left(K(f_1 (x,\beta)) \right) \right] \dhil(\beta, \beta^\sfe) , \\
		&W_{4} (x,\beta) :=  \dhil(K(f_1 (x,\beta))^\top \alpha^\sfe, (K^\sfe)^\top \alpha^\sfe) , \\
		&W_{5} (x,\beta) := \lambda \left(K(f_1(x,\beta)) \right) \dhil(K(x)\beta^\sfe, K^\sfe \beta^\sfe), \\
		&K^\sfe := K(x^\sfe), \ \alpha^\sfe := 1_N /N \oslash [K^\sfe \beta^\sfe].
	\end{align*}
	In the sequel, we explain that sufficiently small $ \varepsilon$ and large $\gamma $ enable us to take a neighborhood $ B_r(x^\sfe, \beta^\sfe) $ where
	\begin{equation}\label{ineq:W<0}
	  W(x,\beta) < 0, \ \forall (x,\beta) \in B_r(x^\sfe, \beta^\sfe)\backslash \{(x^\sfe, \beta^\sfe)\},
	\end{equation}
	which means the asymptotic stability of $ (x^\sfe, \beta^\sfe) $ \cite[Theorem~1.3]{Krabs2010}.

	First, a straightforward calculation yields, for any $ i,j\in \bbra{N}, l\in \bbra{n} $ and any $ (x,\beta) \in \bbR^{nN} \times (\bbR_{>0}^N / {\sim}) $,
	\begin{align*}
		&\left|\frac{\partial}{\partial x_{i,l}} \tilde{P}_{ij} (x, \beta) \right| \le \frac{2N \bar{g}_{i,j,l}}{\varepsilon} \tilde{P}_{ij} (x,\beta) \left( \frac{1}{N}  - \tilde{P}_{ij} (x,\beta) \right), \\
		&x_i = [x_{i,1} \ \cdots \ x_{i,n}]^\top, \\
		&\bar{g}_{i,j,l} := \max_{k\neq j} | g_{i,l}^\top (x_j^\sfd - x_k^\sfd) |, \ \calG_i = [g_{i,1} \ \cdots \ g_{i,n}]^\top .
	\end{align*}
	From Lemma~\ref{lemma:exponential_convergence}, under the assumption $ x_i^\sfd \neq x_j^\sfd, \ i\neq j $, $ \tilde{P}_{ij} (x^\sfe (\varepsilon), \beta^\sfe (\varepsilon)) $ converges exponentially to $ 0 $ or $ 1/N $ as $ \varepsilon \rightarrow +0 $.
	Hence, the variation of $ W_{2,i} $ with respect to $ x $ around $ (x^\sfe(\varepsilon), \beta^\sfe (\varepsilon)) $ can be made arbitrarily small by using sufficiently small $ \varepsilon = \varepsilon_1 $.
	 In addition, since $ \gamma > 0 $ can be chosen independently of $ \varepsilon $, sufficiently large $ \gamma = \bar{\gamma} $ enables us to take a neighborhood $ B_{r_1}(x^\sfe, \beta^\sfe) $ where
	 \begin{align}
		&\sum_{i=1}^N  \left(-\frac{1}{2}W_{1,i} \left(x_i, \tilde{P}(x,\beta) \right) + W_{2,i} (x,\beta)  \right) \nonumber\\
		& + \gamma \left(-\frac{1}{2}W_3(x,\beta) \right) < 0, \ 
	 	 \forall (x,\beta) \in B_{r_1}(x^\sfe, \beta^\sfe) \backslash \{(x^\sfe, \beta^\sfe)\}. \label{ineq:W_part1}
	 \end{align}

	Next, it follows from $ (x^\sfe, \beta^\sfe) \in {\rm Exp}(\sigma) $ that
	\begin{align*}
		\nabla_{x_i} K_{ij}|_{x = x^e (\varepsilon)} &= - \frac{2}{\varepsilon} \exp \left( - \frac{\|x_i^\sfe (\varepsilon) - x_j^\sfd \|_{\calG_i}^2}{\varepsilon} \right) \\
		&\qquad \times \calG_i (x_i^\sfe (\varepsilon) - x_j^\sfd) \rightarrow 0,  \ {\rm as} \ \varepsilon \rightarrow +0  .
	\end{align*}
	 Since $ W_4 $ and $ W_5 $ depend on $ (x,\beta) $ only via $ K $, their variation around $ (x^\sfe (\varepsilon) ,\beta^\sfe (\varepsilon)) $ can be made arbitrarily small by taking sufficiently small $ \varepsilon > 0 $. Therefore, under the assumption that $ (x^\sfe(\varepsilon), \beta^\sfe (\varepsilon)) $ is isolated, for any given $ \gamma > 0 $, we can take $ \varepsilon = \varepsilon_2 (\gamma) $ such that there exists a neighborhood $ B_{r_2} (x^\sfe, \beta^\sfe) $ where
	 \begin{align}
		&\hspace{-0.3cm}\sum_{i=1}^N  \left(-\frac{1}{2}W_{1,i} (x_i, \tilde{P}(x,\beta))\right) + \gamma \biggl(-\frac{1}{2}W_3(x,\beta) + W_4(x,\beta)  \nonumber\\
		&\hspace{-0.3cm}  + W_5(x,\beta) \biggr) < 0, 
	 	 \forall (x,\beta) \in B_{r_2}(x^\sfe, \beta^\sfe) \backslash \{(x^\sfe, \beta^\sfe)\}. \label{ineq:W_part2}
	 \end{align}

	 By combining \eqref{ineq:W_part1} and \eqref{ineq:W_part2}, we obtain \eqref{ineq:W<0} for $ r = \min \{r_1, r_2\} $, $ \gamma = \bar{\gamma} $, and $\varepsilon = \min\{\varepsilon_1, \varepsilon_2 (\bar{\gamma})  \}$, which completes the proof.
\end{proof}

\section{Conclusion}\label{sec:conclusion}
\magenta{
In this paper, we presented the concept of Sinkhorn MPC, which combines MPC and the Sinkhorn algorithm to achieve scalable, cost-effective transport over dynamical systems. 
For linear systems with an energy cost, we analyzed the ultimate boundedness and the asymptotic stability for Sinkhorn MPC based on the stability of the constrained MPC and the conventional Sinkhorn algorithm.

On the other hand, in the numerical example, we observed that the regularization parameter plays a key role in the trade-off between the stationary and transient behaviors for Sinkhorn MPC.
Hence, an important direction for future work is to investigate the design of a time-varying regularization parameter to balance the trade-off.
\rev{Another direction is to extend our results to nonlinear systems with state constraints. In addition, the computational complexity still can be a problem for nonlinear systems. 
These problems will be addressed in future work.
}
}

\addtolength{\textheight}{-12cm}   






\bibliographystyle{IEEEtran}

\bibliography{root}

\end{document}